\documentclass[a4paper,11pt]{amsart}
\usepackage{hyperref,latexsym}
\usepackage{enumerate}

\theoremstyle{plain}
\newtheorem{theorem}{Theorem}[section]

\newtheorem{proposition}[theorem]{Proposition}
\theoremstyle{definition}

\newtheorem{example}{Example}
\theoremstyle{remark}
\newtheorem{remark}{Remark}

\begin{document}

\title[Gutkin Billiards]
      {Gutkin billiard tables in higher dimensions and rigidity}

\date{September 2017}
\author{Misha Bialy }
\address{M. Bialy, School of Mathematical Sciences, Tel Aviv
University, Israel} \email{bialy@post.tau.ac.il}

\thanks{M.B. was supported in part by the Israel Science Foundation grant
162/15}

\subjclass[2010]{} 

\begin{abstract} E. Gutkin found a remarkable class of convex billiard tables in the plane
    which have a constant angle invariant curve. In this paper we prove that in dimension 3
    only round sphere has such a property. 
    For dimension greater than $3$ it must be either a sphere or
    to have a very special geometric properties. 
    In 2-dimensional case we prove a
    rigidity result for Gutkin billiard tables.  This is done with the help of a new generating function introduced recently for billiards in our joint paper with A.E. Mironov. A formula for this generating function in higher dimensions is found.
\end{abstract}

\maketitle



\section{\bf Introduction and main results}
Consider a convex compact domain in Euclidean space $\mathbf{R}^{d}$
bounded by a smooth hypersurface $S$ with positive principal curvatures everywhere.
We shall call $S$ a Gutkin billiard table if
there exists $\delta\in (0;\pi/2)$ such that for any pair of points $p,q \in S $  the following condition is satisfied: if the angle between the vector $\overrightarrow {pq}$ with the tangent hyperplane to $S$ at $p$ equals $\delta$, then the angle between $\overrightarrow {pq}$
and the tangent hyperplane at $q$ also equals $\delta$.

Notice that the case $\delta=\pi/2$ is classical and corresponds to
bodies of constant width.  Planar billiard tables
with this property were found and studied in details by Eugene Gutkin 
\cite{gu1},\cite{gu2} (see also \cite {T2}). He discovered that planar domains with this
property, which are different from round discs, can exist only for
those values of $\delta$ which satisfy for some integer $n>3$ the
equation
\begin {equation}\label{G}
\tan(n\delta)=n\tan \delta.
\end{equation}
Moreover, the shape of these domains is also very special. 
Namely, let $\rho(\phi)$ be the curvature radius as a function of the tangent angle $\phi$. Then the Fourier coefficients $c_k$ of $\rho$ all vanish for $k$ different from $n$ in (\ref{G}).
 For example, if $\delta$ satisfies (\ref{G}) for $n=5$ then
 function $\rho (\phi)=a_0 + a_5\cos5\phi$ gives an example of Gutkin billiard table. Notice that in this example the domain is also of constant width, so for billiard ball map there are two constant angle invariant curves: one for the angle $\delta$ and another for the angle $\pi/2$.

It turns out that the property of equal angles becomes very rigid in higher dimensions:
\begin{theorem}
    The only Gutkin billiard tables in $\mathbf{R}^{3}$ are round spheres.
\end{theorem}
Let us formulate now the result for the case $d>3$. We have the following alternative:
\begin{theorem}Let $S$ be a Gutkin billiard table in $\mathbf{R}^{d}$,  $d>3$, corresponding to an angle $\delta\in (0;\pi/2)$.

1.  If $\delta$ is not a solution of equation (\ref{G}) for odd $n$,
then $S$ is a round sphere.

2. If $\delta$ is a solution of equation (\ref{G}) for some odd $n$, then $S$ is necessarily
a body of constant width. Moreover, every geodesic curve on $S$ which is tangent to a principal direction at some point of $S$ lies in a 2-plane and defines on this plane 2-dimensional Gutkin billiard table.
    \end{theorem}
\begin{remark}
It is plausible that equation (\ref{G}) is in fact irrelevant in higher dimensions, i.e. $S$ must be a sphere also in the case 2. of Theorem 1.2. However we were not able to prove this. 
\end{remark}
The following discussion is important.
It was proved in \cite{ber}, \cite{b-g} that for higher dimensional billiards only ellipsoids have convex caustics. Less restrictive object of dynamical importance would be invariant hypersurface in the phase space. Obviously, if there is a convex caustic for convex billiard table then the set of oriented lines tangent to the caustic form an invariant hypersurface in the phase space of all oriented lines. There are no examples, however of invariant hypersurfaces besides ellipsoids. Gutkin billiards would provide an example of such hypersurface, but Theorem 1.1 and 1.2 tell us that they are very rare. It would be interesting to have an example
and study further properties of invariant hypersurfaces for higher dimensional Birkhoff billiards.

Our next result on Gutkin billiard tables deals with the question of
their integrability in the planar case. 

More precisely, we examine
the so called total integrability in a strip between two neighboring invariant curves which we now turn to
explain. The term of total integrability was suggested in
\cite{knauf} for geodesic flows. Consider a Gutkin billiard table $S$ corresponding to the angle $\delta$. Then $\delta$  is one of the solutions of  equation (\ref{G}), while (\ref{G}) has as many as  $\lfloor \frac{n}{2} \rfloor$ solutions in $[0, \frac{\pi}{2})$. Every solution $\delta_i$ corresponds a constant angle invariant curve on the phase cylinder.
Let me denote by $\Omega_{\delta_1\delta_2}$ the strip between two neighboring constant angle invariant curves on the cylinder.
We shall say that the billiard is totally integrable in the strip $\Omega_{\delta_1\delta_2}$
if the whole strip is foliated by rotational invariant curves.

\begin {theorem}
If Gutkin billiard table $S$ is totally integrable in the strip $\Omega_{\delta_1\delta_2}$ between two neighboring invariant curves then $S$ is a circle.

\end{theorem}

Some discussion is in order. Total integrability, and more generally
Hopf type rigidity for billiards, was  found in \cite{B1} and by
another method in \cite{W} where the assumption was that rotational
invariant curves occupy all (or almost all) phase space.  In \cite{B2} a qualitative version of
Hopf rigidity is obtained. One would
like to relax the conditions of Hopf rigidity in some way. Theorem 1.3 gives a
rigidity result on total integrability for Gutkin billiards on a
strip between two invariant curves.

For the proof of Theorem 1.3 we use a new generating function for Convex billiards invented in \cite{BM}.
We also discuss the formula for this function for higher dimensional case. It turns out that for ellipsoids it coincides with one found by Yu. Suris in \cite{Su}.

In Section 2 we use symplectic nature of the problem and show a link of billiard ball map with geodesics on the surface. Then we prove their planarity.
In Section 3 we prove Theorem 1.1 and 1.2.  Section 4 contains the proof of Theorem 1.3.
In Section 5 we discuss generating function also in higher dimensional case.

\section*{Acknowledgments}
It is a pleasure to thank  Boris Khesin, Yuri Suris and Sergei Tabachnikov for stimulating discussions.
This research was supported in part by ISF grant 162/15.

\section{Higher dimensional Gutkin billiard tables}

\subsection{Symplectic properties}
Proof of Theorem 1.1 requires symplectic properties of billiards.
Consider Birkhoff billiard inside  hypersurface $S$. The phase space
$\Omega$ of the billiard consists of the set of oriented lines
intersecting $S$. The space of oriented lines in $\mathbf{R}^{d}$ is
isomorphic to $T^*\mathbf{S}^{d-1}$ and hence carries natural
symplectic structure. Birkhoff billiard map acts on the space of
oriented lines and preserves this structure. Another way to describe
the same symplectic structure is the following. Every oriented line $l$
intersecting $S$ at $p$ corresponds to a unit vector with foot point
$p$ on $S$. Orthogonal projection onto the tangent space $T_pS$ maps
in 1-1 way sphere of unit vectors with foot point $p$  on $S$ onto unit
ball of the tangent space $ T_pS$. Thus the phase space of oriented
lines intersecting $S$ is isomorphic to unit (co-)ball bundle of
$S$. The canonical symplectic form of this bundle coincides with
that defined above. Here and later we identify co-vectors with
vectors by means of the scalar product induced from
$\mathbf{R}^{d}$.

Using these preliminaries, main observation of this subsection is the following. The fact that $S$ corresponds to Gutkin billiard table with the angle
$\delta$ is equivalent to the fact that the hypersurface  $\Sigma_\delta$ of the phase space $\Omega$ defined by
$$
\Sigma_\delta = \{(p,v)\in \Omega: p\in S, v\in T_p S, |v|=\cos\delta\}
$$
is invariant under the billiard ball map. As a corollary we get the following:
\begin{theorem}
    Given Gutkin billiard table in $\mathbf{R}^{d}$, the billiard ball map transforms every characteristic of $\Sigma_\delta$ to a characteristic.
\end{theorem}
\subsection{Differential geometric interpretation}
Now notice, that since $\Sigma_\delta$ is a bundle of tangent
vectors of constant length, then characteristics of $\Sigma_\delta$
are geodesics equipped with their tangent vectors of the length
$\cos\delta$. Thus the following differential geometric
interpretation can be concluded from Theorem 2.1.

    Given Gutkin billiard table $S\subset\mathbf{R}^{d}$.
    We shall denote by $n(p)$ the unit inner normal vector to $S$ at $p$. Let $\gamma$ be a geodesic curve on $S$. Denote by $s$ the arc length parameter on $\gamma$.
    Let $z(s)$ be the unit vector
    $$
    z(s)=\cos\delta\ \dot{\gamma}(s)+\sin\delta\  n(\gamma(s))
    .$$

    Consider the straight line segment  $[\gamma(s) ;\Gamma(s)]$  such that $\Gamma(s)$ belongs to $S$ and
 $$
    \Gamma(s)- \gamma(s)=l(s)z(s),
    $$
        where $l(s)$ is the length of the segment $[\gamma(s) ;\Gamma(s)]$.
    \begin{theorem}
    It then follows that $\Gamma$ is a geodesic on $S$ with regular parameter $s$ (not necessarily proportional to arc-length). Moreover the following two properties are valid:

    1. The vectors $ \{z(s),\dot{\Gamma}(s), \ddot{\Gamma}(s)\}$ belong to a 2-plane.

    2. The angle between $\dot{\Gamma}(s) $ and $z(s)$ equals precisely $\delta$.
\end{theorem}

\subsection{Deviation from osculating 2-plane}
Notice that since all principal curvatures of $S$ are assumed to be strictly positive, then for any geodesic $\gamma$ on $S$
the curvature $k$ of $\gamma$ in $\mathbf{R}^{d}$ is strictly positive. Therefore we can write first three Frenet equations for $\gamma$ as follows. Denote $v(s)=\dot{\gamma}(s)$, so
\begin{equation}\label{f1}
\dot{v}(s)=k(s)\ n(s).
\end{equation}
Next,  $\dot{n}$ must be orthogonal to $n$ and so belongs to the tangent space to $S$. Hence it can be written as
\begin{equation}
\dot{n}(s)=x\cdot v(s)+ \tau(s) w(s),
\end{equation}
where $w$ is a unit vector in $\mathbf{R}^{d}$ orthogonal to $Span\{v,n\}$.
Differentiating $<v,n>$ along $\gamma$ we get $x\equiv -\tau$.
So
\begin{equation}\label{f2}
\dot{n}(s)=-k(s) v(s)+ \tau(s) w(s).
\end{equation}
In a similar way
$$\frac{d}{ds}\left< w,v\right>=0=\left<\dot w,v\right> + k\left< w,n\right>=\left<\dot w,v\right>,$$
and hence $\dot w$ is orthogonal to $v$ and also to $w$.
Therefore we can write :
\begin{equation}\label{f3}
    \dot{w}(s)=-\tau(s)n(s)+\hat{w}
\end{equation}
where $\hat{w}$ is orthogonal to $Span\{v,n,w\}$.

Notice, that if $d=3$ then $w$ is just a bi-normal vector of $\gamma$, $\hat{w}\equiv0$ and (\ref{f1}), (\ref{f2}),(\ref{f3}) are usual Frenet equations, where $\tau$ is torsion of $\gamma$.
It is important that also in higher dimensions one concludes from (\ref*{f2})  that the function $\tau$ vanishes iff the curve $\gamma$ lies in a 2-plane.

Now we are in position to interpret the conditions 1. and 2.
of Theorem 2.2.
The second condition is easy. Write using Frenet equations (\ref{f1}), (\ref{f2})
\begin{equation}\label{e4}
\dot{\Gamma}=\dot{\gamma}+\dot{l}(\cos\delta v+\sin\delta n)+l
(\cos\delta \ kn+\sin\delta(-kv+\tau w))=
\end{equation}
$$
\ \  =(1+\dot{l}\cos\delta-kl\sin\delta)v+(\dot{l}\sin\delta+kl\cos\delta)n+(\tau l\sin\delta) w.
$$
Using (\ref{e4}) one computes

\begin{equation}\label{e5}
|\dot{\Gamma}|^2=(\dot{l}+\cos\delta)^2+(kl-\sin\delta)^2+\tau^2l^2\sin^2\delta,
\end{equation}
and also
\begin{equation}\label{e6}
\left <\dot{\Gamma},z\right>=\cos\delta(1+\dot{l}\cos\delta-kl\sin\delta)+\sin\delta(\dot{l}\sin\delta+kl\cos\delta)=\dot{l}+\cos\delta.
\end{equation}

Therefore Condition 2. of Theorem 2.2 which reads:
\begin{equation}\label{e7}
\left <\dot{\Gamma},z\right>=|\dot{\Gamma}|\cos\delta.
\end{equation}
takes the form:
\begin{equation}\label{e8}
\dot{l}+\cos\delta=\cos\delta\sqrt{(\dot{l}+\cos\delta)^2+(kl-\sin\delta)^2+\tau^2l^2\sin^2\delta}.
\end{equation}
Simplifying (\ref{e8}) we have:
\begin{equation}\label{e9}
\dot{l}+\cos\delta=\frac{\cos\delta}{\sin\delta}\sqrt{(kl-\sin\delta)^2+\tau^2l^2\sin^2\delta}.
\end{equation}
Since $\dot{\Gamma}$ cannot vanish ($s$ is a regular parameter on $\Gamma$) we can conclude from (\ref{e9}):
\begin{proposition}\label{p1}
    Terms $(kl-\sin\delta)$ and $\tau$ do not vanish simultaneously.
\end{proposition}

\subsection{Differential equation on $\tau$}

 Condition 1. of Theorem 2.2.
says that the vectors $z, \dot{\Gamma}, \ddot{\Gamma}$ are linearly dependent. Denote by  $\pi$ orthogonal projection onto the 3-dimensional space $W=Span\{v, n, w\}$.
Then the
three vectors $\pi(z), \pi( \dot{\Gamma}), \pi( \ddot{\Gamma}) $  must be linearly dependent as well. We need to compute $\ddot{\Gamma}$. Differentiating (\ref{e4}) and using Frenet formulas we get $\pi(\ddot{\Gamma})=a_1v+a_2n+a_3w $:
\begin{equation}\label{e10}
a_1=(\ddot{l}\cos\delta-2k\dot{l})\sin\delta-\dot{k}l\sin\delta-k^2l\cos\delta,
\end{equation}
$$
a_2=1+\dot{l}+2k\dot{l}\cos\delta-k^2l\sin\delta+\ddot{l}\sin\delta+\dot{k}l\cos\delta-\tau^2l\sin\delta,
$$
$$
a_3=2\tau\dot{l}\sin\delta+\tau kl\cos\delta+\dot{\tau}l\sin\delta.
$$

Finally we write the determinant:
$$
D=\det
\begin{Vmatrix}
\cos\delta&\sin\delta&0\\
1+\dot{l}\cos\delta-kl\sin\delta&\dot{l}\sin\delta+kl\cos\delta&\tau l\sin\delta \\
a_1&a_2&a_3
\end{Vmatrix}=0
$$
Fortunately, there is no need to compute this determinant exactly but to notice the following. Every term in the last column  (also terms from $a_3$) contains $\tau$ or $\dot{\tau}$ as a multiplier. Thus $D$ can be written as
\begin{equation}\label{D}
D= A(s)\dot{\tau}+B(s)\tau.
\end{equation}
We don't care about $B(s)$ but need to find $A(s)$ explicitly. But
this is easy because $\dot{\tau}$ is present only in $a_3$. Thus
$$
A(s)=l\sin\delta(\cos\delta(\dot{l}\sin\delta+kl\cos\delta)-\sin\delta(1+\dot{l}\cos\delta-kl\sin\delta))=
$$
$$
=l\sin\delta (kl-\sin\delta).
$$
So we have the following
\begin{theorem}
    Condition 2. of Theorem 2.2 implies the following differential equation on $\tau$:
    $$l\sin\delta (kl-\sin\delta)\dot{\tau}+B(s)\tau=0.$$
    Moreover, if $\tau$ vanishes at one point it must vanish identically.
\end{theorem}
\begin{proof}
Indeed consider the subset $Z$ of $\mathbf{R}$  defined by:
$$
Z=\{s\in\mathbf{R}:\tau(s)=0 \}
$$
 By definition, $Z$ is obviously a closed set. On the other hand, it follows from Proposition 2.3 that if $\tau(s_0)=0$ then $(k(s_0)l(s_0)-\sin\delta)$ does not vanish and hence by the uniqueness for the differential equation, $\tau(s)$ vanishes in
a neighborhood of $s_0$. So $Z$ is an open set. Thus $Z$ coincides with the whole real line.
\end{proof}
\subsection{Planarity of geodesics}
As a consequence of Theorem 2.4 we get planarity of some geodesic curves of $S$.
\begin{theorem}\label{planar}
    Every geodesic curve on $S$ which at some point $p$ passes in a principal direction lies necessarily in a 2-plane
    spanned by this direction and the normal line at $p$. Moreover, this geodesic curve
     has a principal direction at every point where it passes.
\end{theorem}
\begin{proof}Principal directions correspond to the eigenvectors of
    Shape operator. If $\gamma$ has principal direction at $p=\gamma(0)$
    then
    $$\left. \frac{d}{ds} \right|_{s=0} n(\gamma(s))=-k \dot{\gamma}(0).
    $$
    Comparing this formula with Frenet formula (\ref{f2}) we get
    $$\tau(0)=0.$$
     Hence by Theorem 2.4 function $\tau(s)$ vanish identically,
     so the geodesic lies in the two plane. Since at the points of geodesic curve normal to $S$ equals that of the geodesic, then the derivative of the normal satisfies (\ref{f2}) at every point with $\tau=0$ which means that this geodesic has principal direction everywhere on its way.
    \end{proof}
\section{Proof of Theorems 1.1. }
In this Theorem we have $d=3$, so $S$ is two dimensional
and for every point $p\in S$ either $p$ is umbilical or there  are
precisely two orthogonal principal directions.  Let $p$ be a
non-umbilic point, so in a neighborhood of $p$ there are two
orthogonal unit vector fields  $v_1$ and $v_2$ going in principal
directions. Moreover it follows from Theorem \ref{planar} that these
vector fields are orthogonal geodesic vector fields, i.e integral
curves are geodesics. In such a case passing to curvature coordinates on $S$, it is easy to see that
the Riemannian metric  of $S$ must be flat in a
neighborhood of $p$. Indeed in the curvature coordinates $(x,y)$ metric takes the form
$$
ds^2=E(x,y)dx^2+G(x,y)dy^2.
$$
Since $\{x=const\}$ and $\{y=const\}$ are geodesics we get $E=E(x),\ G=G(y)$, but then the metric is flat.
Flatness of the metric yields a contradiction, since $S$ is assumed
to have positive principal curvatures. This argument implies that
all points of $S$ are umbilical and hence $S$ is a round sphere. This
completes the proof in three dimensional case.

\section{Proof of Theorem 1.2. }Now we in the case  $d>3$. Suppose $\gamma$ is a geodesic of $S$ lying in a 2-plane $\sigma$. Then, as a section of the convex hypersurface, $\gamma$ is a convex closed curve in the plane $\sigma$. In addition, since normal to $S$ and normal to $\gamma$ are the same, then the planar billiard inside $\gamma$ is a two dimensional
Gutkin billiard, therefore is very special as I explained in the Introduction.

\subsection{}Let us show now that $S$ has a constant width. Take a point $p\in S$  and any two orthogonal principal directions $v_1, v_2$  at $p$. Then geodesics $\gamma_1, \gamma_2$ in these directions are simple closed convex curves contained in the 2-planes $\sigma_1, \sigma_2$. Intersection of these planes is precisely the normal line $l_p$ through $p$. Therefore the curves $\gamma_1$ and $\gamma_2$ must intersect at a unique point in addition to $p$ which lies on this normal line $l_p$. Indeed, in addition to $p$, $\gamma_1$ must intersect $l_p$ in some other point $p'$, also
$\gamma_2$ must intersect $l_p$ in some point $p''$. Then there are
three points $p,p',p''$ of the intersection of the line $l_p$ with
the convex hypersurface. Therefore, $p'=p''$. Moreover, $l_p$ is
orthogonal to $S$ at $p'$. Indeed, normal to $S$ at $p'$ coincides with
normal to  $\gamma_1$ and to $\gamma_2$ because they are geodesic curves on $S$, thus this normal is parallel to the
intersection of the 2-planes $\sigma_1,\sigma_2$ which is $l_p$. So
we proved the so called double normal property, which is known to be
equivalent to the constant width condition \cite{G-W}. In addition
the plane curves $\gamma_1$ and $\gamma_2$ are of course also of
constant width.

\subsection{} Let us show now that every point of $S$ is umbilical in the case when $\delta$ is not a solution (\ref{G}) for odd $n$. In other words, I claim that all principal curvatures of any point $p$ are all equal to $\frac{1}{R}$, where $R$
is half width of $S$. This implies immediately that $S$ must be a
sphere. In order to prove the claim, take any two principal directions at $p$ and the
geodesics $\gamma_1$ and $\gamma_2$ as above.  Recall that
$\gamma_1$ and $\gamma_2$ are planar Gutkin billiards for the same
angle $\delta$. There are two possibilities. In the first case
$\delta$ is not a solution of (\ref{G}) for any $n$. In this case by
Gutkin result $\gamma_1$ and $\gamma_2$ are circles with the same
diameter. In the second case $\delta$ is a solution of (\ref{G}) for
an even $n$. In such a case \cite{gu1} the Fourier expansion of
$\rho_1(\phi)$, $\rho_2(\phi)$  contains only harmonics of even
multiples of $\phi$. On the other hand we know from the previous
argument that $\gamma_1$ and $\gamma_2$ are of constant width.
Therefore, we have for there curvature radii:
$$
\rho_{1,2}(\phi+\pi)+\rho_{1,2}(\phi)=const.
$$
But this is not possible for functions with even harmonics only.

So we conclude that both $\gamma_1$ and $\gamma_2$ are circles.
Moreover since they have the same diameter, then also the same
curvature $\frac{1}{R}$. This proves the claim and Theorem 1.2.

\section {Rigidity of planar Gutkin billiard tables}
\subsection{Proof of Theorem 1.3} In order to prove Theorem 1.3 we use a new generating function for Birkhoff billiards found in \cite{BM}.
 We fix a point inside $S$ and us the coordinates $(p,\phi)$ on the space of oriented lines intersecting $S$.
 Here $\phi$ is the angle between $x$ axes and the positive  normal to the line, and $p$ is a signed distance to the line. It is proved in \cite{BM} that in these coordinates billiard map is a twist map and can be given with the help of generating function
 \begin{equation}\label{S}
 \rm S=2h\left(\frac{\phi_1+\phi_2}{2}\right) \sin\left(\frac{\phi_2-\phi_1}{2}\right),
 \end{equation}
 where $h(\phi)$ denotes the supporting function of the curve $S$.
 \begin{remark}In \cite{BM}  we used this function near the boundary showing its advantage in proof of KAM type results for billiards. But in fact in the coordinates $(p,\phi)$ billiard ball map satisfies twist condition globally  with
 \begin{equation}\label{s}
 \rm S=2h\left(\frac{\phi_1+\phi_2}{2}\right)\left | \sin\left(\frac{\phi_2-\phi_1}{2}\right)\right|.
 \end{equation}
 Moreover, in higher dimensions billiard ball map is still a twist map in these symplectic coordinates and has a very simple generating function which we shall derive below.
 \end{remark}
 We shall need expressions for second derivatives of $\rm S$:

 Let me denote by $${\phi}=\frac{\phi_1+\phi_2}{2} ;\quad \alpha=\frac{\phi_2-\phi_1}{2}$$

 Then we have following formulas for second partial derivatives of $\rm S $:
 $$
 \rm S_{11}(\phi_1,\phi_2)=\frac{1}{2}(h''(\phi)-h(\phi))\sin \alpha -h'(\phi)\cos \alpha;
 $$
$$
\rm S_{22}(\phi_1,\phi_2)=\frac{1}{2}(h''(\phi)-h(\phi))\sin \alpha +h'(\phi)\cos \alpha;
$$
$$
\rm S_{12}(\phi_1,\phi_2)=\frac{1}{2}(h''(\phi)+h(\phi))\sin \alpha .
$$
From the last formula the twist condition $\rm S_{12}>0$ holds true since  $$h''(\phi)+h(\phi)=\rho(\phi)$$ is the curvature radius.

The following statement follows in a standard way \cite{B1}
\cite{B2}  from the assumption of total integrability:

\begin{proposition}
For any Gutkin billiard table such that $\Omega_{\delta_1\delta_2}$  is foliated by rotational invariant curves the following inequality holds:
\begin{equation}\label{in}
\int_{\Omega_{\delta_1\delta_2}}(\rm S_{11}+2\rm S_{12}+\rm S_{22}) d\mu\leq 0,
\end{equation}
where $d\mu$ is the invariant measure.
\end{proposition}
Next we compute the integral (\ref{in}).

First notice that the invariant measure can be written in the form:
$$
d\mu= \rm S_{12}\ d\phi_1d\phi_2= \frac{1}{2}\rho(\phi)\sin \alpha\  d\phi_1d\phi_2.
$$
Using this and passing from $(\phi_1,\phi_2)$ to $(\phi,\alpha)$
we get the integral

$$
I=\int_{\delta_1}^{\delta_2}\int_0^{2\pi}(\rm S_{11}+2\rm S_{12}+\rm S_{22})\rm S_{12} \ d\phi d\alpha \leq 0.
$$

Substituting the exact expressions for the derivatives and applying Fubini theorem and integration by parts we compute:
$$
I=2\int_{\delta_1}^{\delta_2}\int_0^{2\pi}h''(\phi)(h''(\phi)+h(\phi))\sin^2\alpha \ d\phi d\alpha=
$$
$$
=2\int_{\delta_1}^{\delta_2}\sin^2\alpha\  d\alpha\cdot\int_0^{2\pi}h''(\phi)(h''(\phi)+h(\phi))d\phi =
$$
$$
=2\int_{\delta_1}^{\delta_2}\sin^2\alpha\  d\alpha\cdot\int_0^{2\pi}[(h''(\phi))^2-(h'(\phi))^2]d\phi .
$$

Notice that the last integral is non-negative by Wirtinger inequality applied to the function $h'$. Comparing with
(\ref{in}) we get $I=0$ and hence the equality in the Wirtinger inequality.
But the equality in Wirtinger inequality is possible only for $h'=a\cos\phi+b\sin\phi$ that is $h=h_0+a\sin\phi-b\cos\phi$, which means that
the curve $S$ is a circle. This completes the proof of Theorem 1.3.
\subsection{Formula for Generating function in higher dimensions}
Consider Birkhoff  billiard inside a convex  hypersurface $S$ in
$\mathbf{R}^d$. Billiard ball map $T$ acts on the subset $\Omega$ of
the space of oriented lines intersecting $S$. The latter is
isomorphic to $T^*\mathbf{S}^{d-1}$, every oriented line can be
represented uniquely in the form :
$$
l=\{m+nt\}, |n|=1,\ n\perp m; \quad (m,n)\in T^*\mathbf{S}^{d-1}, n\in\mathbf{S}^{d-1}, m\in T_n^*\mathbf{S}^{d-1}\}
$$
(as before we identify tangent and cotangent spaces).

Consider the Gauss map $$G:S \rightarrow \mathbf{S}^{d-1}, x\mapsto n(x),$$ where $n(x)$ is the outer unit normal to $S$ at $x$. With the help of $G$ it is easy to write supporting function for $S$:
$$
h(n)=\left<G^{-1}n,n  \right>, \quad n\in \mathbf{S}^{d-1}.
$$
\begin{theorem}
    The billiard ball map $T$ can be described with the help of generating function $\rm S$ as follows
    $$
    T :(m_1,n_1)\mapsto (m_2,n_2)\quad \Leftrightarrow \quad m_1=D_1 \rm S,\ m_2=-D_2 \rm S,
    $$
    where $
    \rm S:\mathbf{S}^{d-1}\times\mathbf{S}^{d-1}\setminus\Delta\rightarrow \mathbf{R},
    $ is the function defined by the formula:

    \begin{equation}\label{S}
     \ \rm S(n_1,n_2)=\left< G^{-1}
    \left( \frac{n_1-n_2}{|n_1-n_2|}   \right), n_1-n_2
    \right>=
    \end{equation}
    $$
    =h\left( \frac{n_1-n_2}{|n_1-n_2|}   \right)|n_1-n_2|.
    $$
    Moreover the twist condition is satisfied in the following sense:
    The linear operators
    $$
    D_{12}=D_1\circ D_2:T_{n_1}\mathbf{S}^{d-1} \rightarrow T_{n_2}^*\mathbf{S}^{d-1},\quad D_{21}=D_2\circ D_1:T_{n_2}\mathbf{S}^{d-1} \rightarrow T_{n_1}^*\mathbf{S}^{d-1}
    $$
    are isomorphisms.
    \end{theorem}
\begin{remark}
    1. One can check that the formula (\ref{S}) for $\rm S$ coincides with (\ref{s}) in case $d=2$.

    2. In higher dimensions derivatives of the generating function related to the usual Birkhoff coordinates were
    computed  in \cite{B5} and are also very useful.

        \end{remark}
\begin{proof}
    Given two distinct unit vectors $n_1,n_2$, since $G$ is a diffeomorphism, there is a unique point $P$ on the surface $S$ with
    the normal vector equal $n(P) =\frac{n_1-n_2}{|n_1-n_2|} $.
    Then, by the construction, the straight line $l_1$ coming to $P$ in the direction $n_1$ is reflected to a line $l_2$ in the direction $n_2$.
    We need to compute the derivative of
    $$ \left< G^{-1}\left( \frac{n_1-n_2}{|n_1-n_2|}   \right), n_1-n_2 \right>$$
    along tangent vector $\xi\in T_{n_1}\mathbf{S}^{d-1}$. It has two terms, the first when we differentiate $G^{-1}$ vanishes since $P=G^{-1}\left( \frac{n_1-n_2}{|n_1-n_2|} \right )$ varies in the tangent space to $S$ and $n_1-n_2$ is normal to it. Thus the differential $D_1\rm S $
    acts on $\xi$ by the formula
    $$
    D_1\rm S:  \xi\mapsto\left< G^{-1}\left( \frac{n_1-n_2}{|n_1-n_2|}   \right ), \xi  \right>.
    $$

This functional can be identified with the projection of the vector
$G^{-1}\left( \frac{n_1-n_2}{|n_1-n_2|} \right )$ orthogonally along
$n_1$, i.e

$$
D_1\rm S=G^{-1}\left( \frac{n_1-n_2}{|n_1-n_2|} \right )-\left< G^{-1}\left( \frac{n_1-n_2}{|n_1-n_2|}   \right ), n_1  \right>
n_1.
$$
But the last one is precisely $m_1$, since the line $l_1$ goes through $P$.
Analogously on computes $D_2\rm S $ and also the second differentials.\end{proof}

It turns out that the formula (\ref{S}) has been worked out already
for the case of ellipsoids \cite{Su}. In order to see this we shall
compute $\rm S$ more explicitly:
\begin{example}
    Let $S$ be an ellipsoid in $\mathbf{R}^d$,
    $$
    S=\{ \left<A^{-1}x,x\right>=1    \},
    $$
    given by a positive definite symmetric matrix $A$.
    \end{example}
We compute the Gauss map and the supporting function. Given $n$, we need to find $x\in S$ with $n(x)=n$
we have $$\mu n= A^{-1}x,$$ for some $\mu$.
Then we have $$
 \left<A^{-1}x,x\right>=\left<\mu n, \mu An  \right>=1.
$$
So we have $$\mu=\left<An,  n  \right>^{-\frac{1}{2}}.$$ Therefore $$G^{-1}(n)=\left<An,  n  \right>^{-\frac{1}{2}}An$$
And the supporting function:
$$
h(n)=\left<G^{-1}n,n  \right>=\left<An,  n  \right>^{\frac{1}{2}}.
$$
Thus finally we get
\begin{equation}\label{su}
\rm S(n_1,n_2)=  \left< A\left( \frac{n_1-n_2}{|n_1-n_2|}\right),  \frac{n_1-n_2}{|n_1-n_2|}  \right> ^{\frac{1}{2}}|n_1-n_2|=
\end{equation}
$$= \left<A\left( n_1-n_2\right), n_1-n_2 \right> ^{\frac{1}{2}}.
$$
The last formula coincides with one found in \cite{Su} for the ellipsoid.

\end{document}